\documentclass[12pt,twoside]{amsart}
\usepackage{amssymb,amsmath,amsthm, amscd, enumerate, mathrsfs}
\usepackage{graphicx, hhline}
\usepackage[all]{xy}
\usepackage[usenames]{color}
\usepackage{hyperref}
\hypersetup{colorlinks=true}

\title[Kodaira vanishing theorem]{Kodaira vanishing theorem for log-canonical and 
semi-log-canonical pairs}
\author{Osamu Fujino} 
\date{2014/12/25, version 1.02}

\address{Department of Mathematics, Graduate School of Science, 
Kyoto University, Kyoto 606-8502, Japan}
\email{fujino@math.kyoto-u.ac.jp}

\newcommand{\Exc}[0]{{\operatorname{Exc}}}
\newcommand{\Supp}[0]{{\operatorname{Supp}}}
\newcommand{\Spec}[0]{{\operatorname{Spec}}}
\newcommand{\Sing}[0]{{\operatorname{Sing}}}
\newtheorem{thm}{Theorem}[section]
\newtheorem{lem}[thm]{Lemma}
\newtheorem{cor}[thm]{Corollary}
\newtheorem{prop}[thm]{Proposition}
\newtheorem*{claim}{Claim}

\theoremstyle{definition}
\newtheorem{defn}[thm]{Definition}
\newtheorem{rem}[thm]{Remark}
\newtheorem*{ack}{Acknowledgments} 
\newtheorem{say}[thm]{}
\begin{document}

\maketitle 

\begin{abstract}
We prove the Kodaira vanishing theorem for 
log-canonical and semi-log-canonical pairs. 
We also give a relative vanishing 
theorem of Reid--Fukuda type for 
semi-log-canonical pairs.  
\end{abstract}

\tableofcontents 

\section{Introduction}\label{sec1}

The main purpose of this short paper is to establish: 

\begin{thm}[Kodaira vanishing theorem for 
semi-log-canonical pairs]\label{thm1.1}
Let $(X, \Delta)$ be a projective semi-log-canonical 
pair and let $L$ be an ample Cartier divisor on $X$. 
Then 
$$
H^i(X, \mathcal O_X(K_X+L))=0
$$ 
for every $i>0$. 
\end{thm}

Theorem \ref{thm1.1} is a naive generalization of the Kodaira vanishing 
theorem for semi-log-canonical pairs. As a special case of 
Theorem \ref{thm1.1}, we have: 

\begin{thm}[Kodaira vanishing theorem for 
log-canonical pairs]\label{thm1.2}
Let $(X, \Delta)$ be a projective log-canonical 
pair and let $L$ be an ample Cartier divisor on $X$. 
Then 
$$
H^i(X, \mathcal O_X(K_X+L))=0
$$ 
for every $i>0$. 
\end{thm}

Precisely speaking, we prove the following theorem 
in this paper. Theorem \ref{thm1.3} is a relative version of 
Theorem \ref{thm1.1} and obviously contains Theorem \ref{thm1.1} 
as a special case. 

\begin{thm}[Main theorem]\label{thm1.3}
Let $(X, \Delta)$ be a semi-log-canonical 
pair and let $f:X\to Y$ be a projective morphism 
between quasi-projective varieties. 
Let $L$ be an $f$-ample Cartier divisor 
on $X$. 
Then $$
R^if_*\mathcal O_X(K_X+L)=0
$$ for 
every $i>0$. 
\end{thm}

Although Theorem \ref{thm1.3} has not been stated explicitly 
in the literature, it easily follows from \cite{fujino-slc}, \cite{fujino-vanishing}, 
\cite{fujino-foundation}, and so on. 
In our framework, Theorem \ref{thm1.1} can be seen as a generalization of 
Koll\'ar's vanishing theorem by the theory of mixed Hodge 
structures. 
The statement of Theorem \ref{thm1.1} 
is a naive generalization of 
the Kodaira vanishing theorem.  
However, 
Theorem \ref{thm1.1} is not a simple generalization of the 
Kodaira vanishing theorem from the Hodge-theoretic viewpoint. 

We note the dual form of the Kodaira vanishing theorem 
for Cohen--Macaulay projective semi-log-canonical pairs. 

\begin{cor}[{cf.~\cite[Corollary 6.6]{kss}}]\label{cor1.4}
Let $(X, \Delta)$ be a projective semi-log-canonical 
pair and let $L$ be an ample Cartier divisor on $X$. 
Assume that $X$ is Cohen--Macaulay. 
Then 
$$
H^i(X, \mathcal O_X(-L))=0
$$ 
for every $i<\dim X$. 
\end{cor}

\begin{rem}\label{rem1.5}
The dual form of the Kodaira vanishing theorem, that is, 
$H^i(X, \mathcal O_X(-L))=0$ for 
every ample Cartier divisor $L$ and every $i<\dim X$,  
implies that $X$ is Cohen--Macaulay (see, for example, \cite[Corollary 5.72]{kollar-mori}). 
Therefore, the assumption that $X$ is Cohen--Macaulay in Corollary \ref{cor1.4} is 
indispensable. 
\end{rem}

\begin{rem}\label{rem1.6} 
In \cite[Corollary 6.6]{kss}, Corollary \ref{cor1.4} was obtained for 
{\em{weakly}} semi-log-canonical pairs (see \cite[Definition 4.6]{kss}). 
Therefore, \cite[Corollary 6.6]{kss} is stronger than Corollary \ref{cor1.4}. 
The arguments in \cite{kss} depend on the theory of Du Bois singularities. 
Our approach (see \cite{fujino-unpublished}, 
\cite{fujino-introduction}, \cite{fujino-slc}, \cite{fujino-vanishing}, \cite{fujino-injectivity}, 
\cite{fujino-reid-fukuda},  
\cite{fujino-foundation}, and so on) to various vanishing theorems for 
reducible varieties uses the theory of mixed Hodge structures for cohomology with compact support 
and is different from \cite{kss}. 
\end{rem}

Finally, we note that we can easily generalize 
Theorem \ref{thm1.3} as follows. 

\begin{thm}[Main theorem II]\label{thm1.7}
Let $(X, \Delta)$ be a semi-log-canonical 
pair and let $f:X\to Y$ be a projective morphism 
between quasi-projective varieties. 
Let $L$ be a Cartier divisor 
on $X$ such that $L$ is nef and log big over $Y$ with 
respect to $(X, \Delta)$. 
Then $$
R^if_*\mathcal O_X(K_X+L)=0
$$ for 
every $i>0$. 
\end{thm}

For the definition of nef and log big divisors on semi-log-canonical 
pairs, see Definition \ref{def2.3}. 
Theorem \ref{thm1.7} is a relative vanishing theorem of Reid--Fukuda 
type for semi-log-canonical pairs. 
It is obvious that Theorem \ref{thm1.1}, Theorem \ref{thm1.2}, and 
Corollary \ref{cor1.4} hold true under the weaker assumption that 
$L$ is nef and log big with respect to $(X, \Delta)$ by Theorem \ref{thm1.7}. 

\begin{ack}
The author was partially supported by the Grant-in-Aid for 
Young Scientists (A) 24684002 and Grant-in-Aid for 
Scientific Research (S) 24224001 from JSPS. 
He thanks Professor J\'anos Koll\'ar. 
This short paper is an answer to his question. 
\end{ack}

Throughout this paper, we will work over $\mathbb C$, the 
field of complex numbers. 
We will use the basic definitions and the standard notation of 
the minimal model program and semi-log-canonical pairs 
in \cite{fujino-funda}, \cite{fujino-slc}, \cite{fujino-foundation}, 
and so on. 

\section{Preliminaries}\label{sec2} 
In this section, we quickly recall some basic definitions and 
results for semi-log-canonical pairs for the reader's convenience. 
Throughout this paper, a variety means a reduced separated scheme of finite 
type over $\mathbb C$. 

\begin{say}[$\mathbb R$-divisors]\label{say2.1}
Let $D$ be an $\mathbb R$-divisor on an equidimensional variety $X$, that is, $D$ is 
a finite formal $\mathbb R$-linear combination 
$$
D=\sum _i d_i D_i
$$ 
of irreducible reduced subschemes $D_i$ of codimension 
one. Note that $D_i\ne D_j$ for $i\ne j$ and that $d_i \in 
\mathbb R$ for every $i$. 
For every real number $x$, $\lceil x\rceil$ is the integer defined by $x\leq 
\lceil x\rceil <x+1$. 
We put $\lceil D\rceil =\sum _i \lceil d_i \rceil D_i$ and 
$D^{<1}=\sum _{d_i <1}d_i D_i$. 
We call $D$ a boundary (resp.~subboundary) $\mathbb R$-divisor 
if $0\leq 
d_i \leq 1$ (resp.~$d_i \leq 1$) for 
every $i$. 
\end{say}

Let us recall the definition of semi-log-canonical 
pairs. 

\begin{defn}[Semi-log-canonical pairs]\label{def2.2}
Let $X$ be an equidimensional variety that
satisfies Serre's $S_2$ condition and is normal crossing in 
codimension one. Let $\Delta$ be an effective
$\mathbb R$-divisor 
such that no irreducible components of $\Delta$ are contained 
in the singular locus of $X$. 
The pair $(X, \Delta)$ is called a semi-log-canonical pair if 
\begin{itemize}
\item[(1)] $K_X + \Delta$ is $\mathbb R$-Cartier, and
\item[(2)] $(X^\nu, \Theta)$ is log-canonical, where $\nu:X^\nu\to X$ is the normalization 
and $K_{X^\nu}+\Theta=\nu^*(K_X+\Delta)$. 
\end{itemize}
A subvariety $W$ of $X$ is called an slc stratum with respect to 
$(X, \Delta)$ if there exist a resolution of singularities $\rho:Z\to X^\nu$ and 
a prime divisor $E$ on $Z$ such that 
$a(E, X^\nu, \Theta)=-1$ and $\nu\circ \rho (E)=W$ or if 
$W$ is an irreducible component of $X$. 
\end{defn}

For the basic definitions and properties of 
log-canonical pairs, see \cite{fujino-funda}. 
For the details of semi-log-canonical pairs, see \cite{fujino-slc}. 
We need the notion of nef and log big divisors on semi-log-canonical 
pairs for Theorem \ref{thm1.7} 

\begin{defn}[Nef and log big divisors on semi-log-canonical pairs]\label{def2.3}
Let $(X, \Delta)$ be a semi-log-canonical pair and let $f:X\to Y$ be 
a projective morphism between quasi-projective varieties. 
Let $L$ be a Cartier divisor on $X$. 
Then $L$ is nef and log big over $Y$ with respect to $(X, \Delta)$ 
if $L$ is nef and $\mathcal O_X(L)|_W$ 
is big over $Y$ for every slc stratum $W$ of $(X, \Delta)$. 
We simply say that $L$ is nef and log big with respect to $(X, \Delta)$ when 
$Y=\Spec \mathbb C$. 
\end{defn}

Roughly speaking, in \cite{fujino-slc}, we proved the following theorem. 

\begin{thm}[{see \cite[Theorem 1.2 and Remark 1.5]{fujino-slc}}]\label{thm2.4}
Let $(X, \Delta)$ be a quasi-projective semi-log-canonical pair. 
Then we can construct a smooth quasi-projective variety $M$ 
with $\dim M=\dim X+1$, 
a simple normal crossing divisor $Z$ on $M$, 
a subboundary $\mathbb R$-divisor 
$B$ on $M$, and 
a projective surjective morphism 
$h:Z\to X$ with the following properties. 
\begin{itemize}
\item[(1)] $B$ and $Z$ have no common irreducible components. 
\item[(2)] $\Supp (Z+B)$ is a simple normal crossing divisor on $M$. 
\item[(3)] $K_Z+\Delta_Z\sim _{\mathbb R}h^*(K_X+\Delta)$ such that $\Delta_Z=B|_Z$. 
\item[(4)] $h_*\mathcal O_Z(\lceil -\Delta_Z^{<1}\rceil)\simeq \mathcal O_X$. 
\end{itemize}
By the properties $(1)$, $(2)$, $(3)$, and $(4)$, $[X, K_X+\Delta]$ has 
a quasi-log structure with only qlc singularities. 
Furthermore, if the irreducible components of $X$ have no self-intersection in 
codimension one, then we can make $h:Z\to X$ birational. 
\end{thm}

For the details of Theorem \ref{thm2.4}, see \cite{fujino-slc}. 
In this paper, we do not discuss quasi-log schemes. 
For the theory of quasi-log schemes, 
see \cite{fujino-introduction}, \cite{fujino-pull}, \cite{fujino-foundation}, and so on. 

\begin{rem}\label{rem2.5}
The morphism $h:(Z, \Delta_Z)\to X$ in Theorem \ref{thm2.4} is called a quasi-log resolution. 
Note that the quasi-log structure of $[X, K_X+\Delta]$ obtained 
in Theorem \ref{thm2.4} is compatible with 
the original semi-log-canonical structure of $(X, \Delta)$. 
For the details, see \cite{fujino-slc}. 
We also note that we have to know how to construct $h:Z\to X$ in \cite[Section 
4]{fujino-slc} 
for the proof of Theorem \ref{thm1.3}. 
\end{rem}

We note the notion of simple normal crossing pairs. 
It is useful for our purposes in this paper. 

\begin{defn}[Simple normal crossing 
pairs]\label{def2.6}Let $Z$ be a simple normal crossing divisor 
on a smooth 
variety $M$ and let $B$ be an $\mathbb R$-divisor 
on $M$ such that 
$\Supp (B+Z)$ is a simple normal crossing divisor and that 
$B$ and $Z$ have no common irreducible components. 
We put $\Delta_Z=B|_Z$ and consider the pair $(Z, \Delta_Z)$. 
We call $(Z, \Delta_Z)$ a globally embedded simple normal 
crossing pair. 
A pair $(Y, \Delta_Y)$ is called a simple normal crossing 
pair if it is Zariski locally isomorphic to 
a globally embedded simple normal crossing 
pair. 

If $(X ,0)$ is a simple normal crossing pair, then 
$X$ is called a simple normal crossing variety. 
Let $X$ be a simple normal crossing variety and 
let $D$ be a Cartier divisor on $X$. 
If $(X, D)$ is a simple normal crossing pair and 
$D$ is reduced, 
then $D$ is called a simple normal crossing divisor on $X$. 
\end{defn}

For the details of 
simple normal crossing pairs, 
see \cite[Definition 2.8]{fujino-slc}, 
\cite[Definition 2.6]{fujino-vanishing}, 
\cite[Definition 2.6]{fujino-injectivity}, 
\cite[Definition 2.4]{fujino-pull}, 
\cite[5.2.~Simple normal crossing pairs]{fujino-foundation}, and so on. 
We note that a simple 
normal crossing pair is called {\em{semi-snc}} in 
\cite[Definition 1.10]{kollar} (see also \cite[Definition 1.1]{bierstone-p}) 
and 
that a globally 
embedded simple 
normal crossing pair is called an {\em{embedded semi-snc pair}} 
in \cite[Definition 1.10]{kollar}. 

\section{Proof of Theorem \ref{thm1.3}}\label{sec3}

In this section, we prove Theorem \ref{thm1.3} and 
discuss some related results. 

Let us start with an easy lemma. 
The following lemma is more or less well-known to 
the experts. 

\begin{lem}[{\cite[Lemma 3.15]{kss}}]\label{lem3.1}
Let $X$ be a normal irreducible variety and let $\Delta$ be 
an effective $\mathbb R$-divisor 
on $X$ such that 
$(X, \Delta)$ is log-canonical. 
Let $\rho:Z\to X$ be a proper birational morphism 
from a smooth variety $Z$ such that 
$E=\Exc(\rho)$ and $\Exc(\rho)\cup \Supp f^{-1}_*\Delta$ are simple 
normal crossing divisors on $Z$. 
Let $S$ be an integral divisor on $X$ such that 
$0\leq S\le \Delta$ and let $T$ be the strict transform of $S$. 
Then we have $\rho_*\mathcal O_Z(K_Z+T+E)\simeq 
\mathcal O_X(K_X+S)$. 
\end{lem} 
We give a proof of Lemma \ref{lem3.1} here for the reader's convenience. 
The following proof is in \cite{kss}. 

\begin{proof}
We choose $K_Z$ and $K_X$ satisfying $\rho_*K_Z=K_X$. 
It is obvious that $\rho_*\mathcal O_Z(K_Z+T+E)\subset 
\mathcal O_X(K_X+S)$ since $E$ is $\rho$-exceptional 
and $\mathcal O_X(K_X+S)$ satisfies Serre's $S_2$ condition. 
Therefore, it is sufficient to prove that 
$\mathcal O_X(K_X+S)\subset \rho_*\mathcal O_Z(K_Z+T+E)$. 
Note that we may assume that $\Delta$ is an effective $\mathbb Q$-divisor 
by perturbing the coefficients of $\Delta$ slightly. 
Let $U$ be any nonempty Zariski open set of $X$. 
We will see that $\Gamma (U, \mathcal O_X(K_X+S))\subset 
\Gamma (U, \rho_*\mathcal O_Z(K_Z+T+E))$. 
We take a nonzero rational function $g$ of $U$ such that 
$\left((g)+K_X+S\right)|_U\geq 0$, that is, $g\in \Gamma (U, \mathcal O_X(K_X+S))$, 
where $(g)$ is the principal divisor associated to $g$. 
We assume that $U=X$ by shrinking $X$ for simplicity. 
Let $a$ be a positive integer such that 
$a(K_X+\Delta)$ is Cartier. 
We have $\rho^*(a(K_X+\Delta))=aK_Z+a\Delta'+\Xi$, where  
$\Delta'$ is the strict transform of $\Delta$ and $\Xi$ is a $\rho$-exceptional 
integral divisor on $Z$. 
By assumption, we have $0\leq (g)+K_X+S\leq (g)+K_X+\Delta$. 
Then we obtain that 
\begin{align*}
0&\leq (\rho^*g^a)+\rho^*(aK_X+a\Delta)\\ &\leq a\left((\rho^*g)+K_Z+\Delta'+E\right) 
\end{align*} since $\Xi \leq a E$. 
Thus we obtain $(\rho^*g)+K_Z+\Delta'+E\geq 0$. 
\begin{claim}$(\rho^*g)+K_Z+T+E\geq 0$. 
\end{claim}
\begin{proof}[Proof of Claim]
By construction, 
$$
(\rho^*g)+K_Z+T+E= \rho^{-1}_*\left( (g)+K_X+S\right) +F+E, 
$$ 
where every irreducible component of $F+E$ is $\rho$-exceptional. 
We also have 
$$
(\rho^*g)+K_Z+T+E=(\rho^*g)+K_Z+\Delta'+E-(\Delta'-T), 
$$ 
where $\Delta'-T$ is effective and no irreducible components 
of $\Delta'-T$ are $\rho$-exceptional.  
Note that $\rho^{-1}_*\left( (g)+K_X+S\right)\geq 0$ and 
$(\rho^*g)+K_Z+\Delta'+E\geq 0$. 
Therefore, we have $(\rho^*g)+K_Z+T+E\geq 0$. 
\end{proof}
This means that $\Gamma (U, \mathcal O_X(K_X+S))\subset 
\Gamma (U, \rho_*\mathcal O_Z(K_Z+T+E))$ for any 
nonempty Zariski open set 
$U$. 
Thus, we have $\mathcal O_X(K_X+S)=\rho_*\mathcal O_Z(K_Z+T+E)$. 
\end{proof}

We need the following remark for the proof of Theorem \ref{thm1.7} in Section \ref{sec4}. 

\begin{rem}\label{rem3.2} 
In Lemma \ref{lem3.1}, we put $$
E'=\sum E_i
$$ where $E_i$'s are the $\rho$-exceptional divisors with 
$a(E_i, X, \Delta)=-1$. 
Then we see that 
$\rho_*\mathcal O_Z(K_Z+T+E')\simeq \mathcal O_X(K_X+S)$ 
by the proof of Lemma \ref{lem3.1}.    
\end{rem}

Although Theorem \ref{thm1.2} is a special case of 
Theorem \ref{thm1.1} and Theorem \ref{thm1.3}, 
we give a simple proof of Theorem \ref{thm1.2} for the 
reader's convenience. 
For this purpose, 
let us recall an easy generalization of Koll\'ar's vanishing theorem. 

\begin{thm}[{\cite[Theorem 2.6]{fujino-higher}}]\label{thm3.3}
Let $f:V\to W$ be a morphism 
from a smooth projective variety $V$ onto a projective 
variety $W$. 
Let $D$ be a simple normal crossing divisor on $V$. 
Let $H$ be an ample Cartier divisor on $W$. 
Then $H^i(W, \mathcal O_W(H)\otimes R^jf_*\mathcal O_V(K_V+D))=0$ for 
$i>0$ and $j\geq 0$. 
\end{thm}

For the proof, see \cite[Theorem 2.6]{fujino-higher} 
(see also \cite{fujino-on-injectivity}, 
\cite[Sections 5 and 6]{fujino-funda}, and so on). 
If $D=0$ in Theorem \ref{thm3.3}, then Theorem \ref{thm3.3} is 
nothing but Koll\'ar's vanishing theorem. 
For more general results, see \cite{fujino-on-injectivity}, \cite{fujino-funda}, 
and so on (see also Theorem \ref{thm3.7} below, \cite{fujino-vanishing}, 
\cite[Chapter 5]{fujino-foundation}, and so on, for vanishing theorems for 
reducible varieties). 

Let us start the proof of Theorem \ref{thm1.2} 
(see \cite[Corollary 2.9]{fujino-introduction} when $\Delta=0$). 

\begin{proof}[Proof of Theorem \ref{thm1.2}]
We take a projective birational morphism $\rho:Z\to X$ from 
a smooth projective 
variety $Z$ such that $E=\Exc(\rho)$ and 
$\Exc(\rho)\cup \Supp \rho^{-1}_*\Delta$ are simple normal crossing 
divisors on $Z$. 
By Theorem \ref{thm3.3}, 
we obtain that $H^i(X, \mathcal O_X(L)\otimes \rho_*\mathcal O_Z(K_Z+E))=0$ 
for every $i>0$. 
By Lemma \ref{lem3.1}, 
$\rho_*\mathcal O_Z(K_Z+E)\simeq 
\mathcal O_X(K_X)$. 
Therefore, 
we have $H^i(X, \mathcal O_X(K_X+L))=0$ for 
every $i>0$.  
\end{proof}

The following key proposition for the proof of Theorem \ref{thm1.3} 
is a generalization of 
Lemma \ref{lem3.1}. 

\begin{prop}\label{prop3.4}
Let $(X, \Delta)$ be a quasi-projective 
semi-log-canonical 
pair such that 
the irreducible components 
of $X$ have no self-intersection in codimension one. 
Then there exist a birational quasi-log resolution 
$h:(Z, \Delta_Z)\to X$ from a globally embedded simple normal crossing 
pair $(Z, \Delta_Z)$ and a simple 
normal crossing divisor $E$ on $Z$ such that 
$h_*\mathcal O_Z(K_Z+E)\simeq \mathcal O_X(K_X)$. 
\end{prop}

\begin{proof}
Since $X$ is quasi-projective and the irreducible components of $X$ have 
no self-intersection in codimension one, we can construct a birational 
quasi-log resolution $h:(Z, \Delta_Z)\to X$ by \cite[Theorem 1.2 and Remark 
1.5]{fujino-slc} (see Theorem \ref{thm2.4}), 
where $(Z, \Delta_Z)$ is a globally embedded simple normal 
crossing pair and the ambient space $M$ of $(Z, \Delta_Z)$ is a smooth 
quasi-projective 
variety. 
By the construction of 
$h:Z\to X$ 
in \cite[Section 4]{fujino-slc}, $\Sing Z$, the 
singular locus of $Z$, maps birationally onto 
the closure of $\Sing X^{\mathrm{snc2}}$, where $X^{\mathrm{snc2}}$ is the open 
subset of $X$ which has only smooth points and simple normal crossing 
points of multiplicity $\leq 2$. 
We put $E=\Exc(h)$. 
Note that $E$ contains no irreducible components 
of $\Sing Z$ by construction.  
By taking a blow-up of $Z$ along $E$ and a suitable birational 
modification (see \cite[Theorem 1.4]{bierstone-p}), 
we may assume that $E$ and $E\cup \Supp h^{-1}_*\Delta$ are 
simple normal crossing divisors on $Z$. 
In particular, $(Z, E)$ is a simple normal crossing 
pair (see Definition \ref{def2.6}). 
Note that \cite[Section 8]{fujino-pull} may help us understand how to 
make $(Z, \Delta_Z)$ a globally embedded simple normal crossing pair. 
We may assume that 
the support of $K_Z$ does not contain any irreducible components of $\Sing Z$ since 
$Z$ is quasi-projective. 
We may also assume that $h_*K_Z=K_X$. 
Then we have $h_*\mathcal O_Z(K_Z+E)\subset \mathcal O_X(K_X)$ since 
$\mathcal O_X(K_X)$ satisfies Serre's $S_2$ condition and 
$E$ is $h$-exceptional. 
We fix an embedding $\mathcal O_Z(K_Z+E)\subset \mathcal K_Z$, 
where $\mathcal K_Z$ is the sheaf of 
total quotient rings of $\mathcal O_Z$. 
Note that 
$h:Z\setminus E\to X\setminus h(E)$ is an isomorphism. 
We put $U=X\setminus h(E)$ and consider the natural open immersion 
$\iota:U\hookrightarrow X$. 
Then we have an embedding $\mathcal O_X(K_X)\subset 
\mathcal K_X$, where $\mathcal K_X$ is the sheaf of total quotient rings of 
$\mathcal O_X$, by $\mathcal O_X(K_X)=\iota_*\bigl ( h_* \mathcal O_Z(K_Z+E)|_U
\bigr)\subset \iota_*\mathcal K_U=\mathcal K_X\bigl(=h_*\mathcal K_Z\bigr)$. 
Let $\nu_X: X^\nu\to X$ be the normalization and let $\mathcal C_{X^\nu}$ be the 
divisor on $X^\nu$ defined by the conductor ideal $\mathfrak {cond}_X$ of 
$X$ (see, for example, \cite[Definition 2.1]{fujino-slc}). Then we have 
$\mathcal O_X(K_X)\subset (\nu_X)_*\mathcal O_{X^\nu}(K_{X^\nu}+\mathcal 
C_{X^\nu})$. 
We put $K_{X^\nu}+\Theta =\nu_X^*(K_X+\Delta)$. 
Then $0\leq \mathcal C_{X^{\nu}}\leq \Theta$  and $(X^\nu, \Theta)$ is log-canonical 
by definition. 
Let $\nu_Z: Z^\nu\to Z$ be the normalization. 
Thus we have $K_{Z^\nu}+\mathcal C_{Z^{\nu}}=\nu_Z^*K_Z$, 
where $\mathcal C_{Z^\nu}$ is the simple normal crossing divisor 
on $Z^\nu$ defined by the conductor ideal $\mathfrak{cond}_Z$ of $Z$. 
Now we have the following commutative diagram. 
$$
\xymatrix{
X^\nu \ar[d]_{\nu_X}& Z^\nu\ar[d]^{\nu_Z}\ar[l]_{h^\nu} \ar[dl]^\varphi\\ 
X & Z \ar[l]^h
}
$$ 
By Lemma \ref{lem3.1} and its proof, 
we see that $\mathcal O_{X^\nu}(K_{X^\nu}+\mathcal C_{X^\nu})
=h^\nu_*\mathcal 
O_{Z^\nu}(K_{Z^\nu}+\mathcal C_{Z^\nu}+\nu_Z^*E)$. 
Therefore, we obtain 
\begin{align*}
\mathcal O_X(K_X)&\subset \varphi_*\mathcal O_{Z^\nu}(K_{Z^\nu}+\mathcal C_{Z^\nu}
+\nu_Z^*E)\\ 
&=\varphi_*\mathcal O_{Z^\nu}(\nu_Z^*(K_Z+E)). 
\end{align*}
This implies that $\mathcal O_X(K_X)\subset h_*\mathcal O_Z(K_Z+E)$. 
Note that $h:Z\setminus E\to X\setminus h(E)$ is an 
isomorphism. 
Thus, we obtain $\mathcal O_X(K_X)=h_*\mathcal O_Z(K_Z+E)$ since 
$h_*\mathcal O_Z(K_Z+E)\subset \mathcal O_X(K_X)$. 
\end{proof}

\begin{rem}\label{rem3.5}
For the details of $\mathcal K_Z$ and $\mathcal K_X$, we recommend the 
reader to see the paper-back edition of \cite[Section 7.1]{liu} published in 2006 
(see also \cite{kleiman}). 
Note that the sheaf of total quotient rings is called 
the sheaf of stalks of meromorphic functions in \cite{liu}. 
\end{rem}

\begin{rem}\label{rem3.6}
As in Remark \ref{rem3.2}, in 
Proposition \ref{prop3.4}, 
we put 
$$
E'=\sum E_i
$$ 
where $E_i$'s are the $h$-exceptional divisors with 
the discrepancy coefficient 
$a(E_i, X, \Delta)\bigl(=a(E_i, X^\nu, \Theta) \bigr)=-1$. 
Then we have $h_*\mathcal O_Z(K_Z+E')\simeq \mathcal O_X(K_X)$ in 
Proposition \ref{prop3.4}. 
This easily follows from Remark \ref{rem3.2} and the proof of Proposition \ref{prop3.4}. 
\end{rem} 

For the proof of Theorem \ref{thm1.3}, 
we use the following vanishing theorem, which is 
obviously a generalization of Theorem \ref{thm3.3}. 
For the proof, see \cite[Theorem 1.1]{fujino-vanishing} 
(see also \cite[Chapter 5]{fujino-foundation}). 

\begin{thm}[{\cite{fujino-unpublished}, 
\cite[Theorem 1.1]{fujino-vanishing}, \cite{fujino-foundation}, and so on}]\label{thm3.7} 
Let $(Z, C)$ be a simple normal crossing pair such that 
$C$ is a boundary $\mathbb R$-divisor on $Z$. 
Let $h: Z\to X$ be a proper morphism to a variety $X$ 
and let $f:X\to Y$ be a projective morphism 
to a variety $Y$. 
Let $D$ be a Cartier divisor on $Z$ such that 
$D-(K_Z+C)\sim _{\mathbb R}f^*H$ for some 
ample $\mathbb R$-divisor $H$ on $X$. 
Then we have $R^if_*R^jh_*\mathcal O_Z(D)=0$ for every $i>0$ and $j\geq 0$. 
\end{thm}

Let us start the proof of Theorem \ref{thm1.3}. 

\begin{proof}[Proof of Theorem \ref{thm1.3}]
We take a natural finite double cover $p:\widetilde X\to X$ due to 
Koll\'ar (see \cite[Lemma 5.1]{fujino-slc}), which is \'etale in codimension one. 
Since $K_{\widetilde X}+\widetilde \Delta=p^*(K_X+\Delta)$ is semi-log-canonical and 
$\mathcal O_X(K_X)$ is a direct summand of 
$p_*\mathcal O_{\widetilde X}(K_{\widetilde 
X})$, we may assume that 
the irreducible components of $X$ have no self-intersection in codimension one 
by replacing $(X, \Delta)$ with $(\widetilde X, \widetilde \Delta)$. 
By Proposition \ref{prop3.4}, 
we can take a birational 
quasi-log resolution $h:(Z, \Delta_Z)\to X$ from 
a globally embedded simple normal crossing pair $(Z, \Delta_Z)$ such that 
there exists a simple normal crossing divisor $E$ on $Z$ satisfying 
$h_*\mathcal O_Z(K_Z+E)\simeq \mathcal O_X(K_X)$. 
Note that $K_Z+E+h^*L-(K_Z+E)=h^*L$. 
Therefore, 
we obtain that 
\begin{align*}
&R^if_*\mathcal O_X(K_X+L) \\ &\simeq 
R^if_*\left(h_*\mathcal O_Z(K_Z+E)\otimes \mathcal O_X(L)\right)=0
\end{align*} 
for every $i>0$ by Theorem \ref{thm3.7}. 
\end{proof}

\begin{rem}\label{rem3.8} 
If $\Delta=0$ in Theorem \ref{thm1.3}, 
then Theorem \ref{thm1.3} follows from \cite[Theorem 1.7]{fujino-slc}. 
Note that the formulation of \cite[Theorem 1.7]{fujino-slc} seems to 
be more useful for some applications than the formulation of Theorem 
\ref{thm1.3}. 

Let $(X, \Delta)$ be a semi-log-canonical Fano variety, that is, 
$(X, \Delta)$ is a projective semi-log-canonical 
pair such that $-(K_X+\Delta)$ is ample (see \cite[Section 6]{fujino-pull}). 
Then $H^i(X, \mathcal O_X)=0$ for 
every $i>0$ by \cite[Theorem 1.7]{fujino-slc}. 
Unfortunately, this vanishing result for 
semi-log-canonical Fano varieties does not 
follow from Theorem \ref{thm1.1}. 
See also Remark \ref{rem3.10} below. 
\end{rem}

Let us prove Theorem \ref{thm1.1} and Theorem \ref{thm1.2}. 

\begin{proof}[Proof of Theorem \ref{thm1.1}]
Theorem \ref{thm1.1} is a special case of Theorem \ref{thm1.3}. 
By putting $Y=\Spec \mathbb C$ in Theorem \ref{thm1.3}, 
we obtain Theorem \ref{thm1.1}.  
\end{proof}

\begin{proof}[Proof of Theorem \ref{thm1.2}]
If $(X, \Delta)$ is log-canonical, then $(X, \Delta)$ is semi-log-canonical. 
Therefore, Theorem \ref{thm1.2} is 
contained in Theorem \ref{thm1.1}. 
\end{proof}

As a direct easy application of Theorem \ref{thm1.1}, we have: 

\begin{cor}\label{cor3.9} 
Let $X$ be a stable variety, that is, $X$ is a projective semi-log-canonical 
variety such that $K_X$ is ample. 
Then $$H^i(X, \mathcal O_X((1+ma)K_X))=0$$ for 
every $i>0$ and every positive integer $m$, 
where $a$ is a positive integer such that $aK_X$ is Cartier. 
\end{cor}

\begin{rem}\label{rem3.10} 
Let $X$ be a stable variety as in Corollary \ref{cor3.9}. 
By \cite[Corollary 1.9]{fujino-slc}, 
we have already known that $H^i(X, \mathcal O_X(mK_X))=0$ for 
every $i>0$ and every positive integer $m\geq 2$. 
This is an easy consequence of \cite[Theorem 1.7]{fujino-slc}. 
\end{rem}

Finally, we prove Corollary \ref{cor1.4}. 

\begin{proof}[Proof of Corollary \ref{cor1.4}]
Since $X$ is Cohen--Macaulay, we see that 
the vector space $H^i(X, \mathcal O_X(-L))$ is 
dual to 
$H^{\dim X-i}(X, \mathcal O_X(K_X+L))$ by Serre duality. 
Therefore, we have $H^i(X, \mathcal O_X(-L))=0$ for 
every $i<\dim X$ by Theorem \ref{thm1.1}. 
\end{proof}

\begin{rem}\label{rem3.11} 
The approach to the Kodaira vanishing theorem 
explained in \cite[Section 6]{kss} can not be directly 
applied to non-Cohen--Macaulay varieties. 
The above proof of Corollary \ref{cor1.4} is different from 
the strategy in \cite[Section 6]{kss}. 
\end{rem}

\section{Proof of Theorem \ref{thm1.7}}\label{sec4} 

In this final section, we just explain how to modify the proof of 
Theorem \ref{thm1.3} in order to obtain Theorem \ref{thm1.7}. 
We do not explain a generalization of Theorem \ref{thm3.7} for 
nef and log big divisors (see \cite[Theorem 5.7.3]{fujino-foundation}), 
which is a main ingredient of the proof of Theorem \ref{thm1.7} 
below. 

Let us start the proof of Theorem \ref{thm1.7}. 

\begin{proof}[Proof of Theorem \ref{thm1.7}] 
Let $p:\widetilde X\to X$ be 
a natural finite double cover as in the proof of Theorem \ref{thm1.3}. 
Note that $p^*L$ is nef and log big over $Y$ with 
respect to $(\widetilde X, \widetilde \Delta)$. 
Therefore, we may assume that the irreducible components of $X$ have 
no self-intersection in codimension one by replacing $(X, \Delta)$ with 
$(\widetilde X, \widetilde \Delta)$. 
We take a birational quasi-log resolution $h:(Z, \Delta_Z)\to X$ as 
in Proposition \ref{prop3.4}. Let $E'$ be the divisor defined in Remark \ref{rem3.5}. 
In this case, $L$ is nef and log big over $Y$ with respect to 
$h:(Z, E')\to X$ (see \cite[Definition 5.7.1]{fujino-foundation}). 
Then we obtain that 
\begin{align*}
&R^if_*\mathcal O_X(K_X+L) \\ &\simeq 
R^if_*\left(h_*\mathcal O_Z(K_Z+E')\otimes \mathcal O_X(L)\right)=0
\end{align*} 
for every $i>0$ by \cite[Theorem 5.7.3]{fujino-foundation} 
(see also \cite[Theorem 2.47 (ii)]{fujino-unpublished} and 
\cite[Theorem 6.3 (ii)]{fujino-fujisawa}). Note that 
$K_Z+E'+h^*L-(K_Z+E')=h^*L$ and that the $h$-image of 
any stratum of $(Z, E')$ is an slc stratum of $(X, \Delta)$ by construction 
(see Definition \ref{def2.2}). 
\end{proof}

\begin{rem}\label{rem4.1}
For the details of 
the vanishing theorem for 
nef and log big divisors and some related topics, 
see \cite[5.7.~Vanishing theorems of Reid--Fukuda type]
{fujino-foundation}. 
Note that \cite{fujino-foundation} is a completely revised and expanded version of 
the author's unpublished manuscript \cite{fujino-unpublished}. 
\end{rem}

\begin{rem}\label{rem4.2} 
We strongly recommend the reader to see Theorem 1.10, 
Theorem 1.11, and Theorem 1.12 in \cite{fujino-slc}. 
They are useful and powerful vanishing theorems 
for semi-log-canonical pairs related to Theorem \ref{thm1.7}. 
\end{rem}


\begin{thebibliography}{KSS}

\bibitem[BP]{bierstone-p} 
E.~Bierstone, F.~Vera Pacheco, 
Resolution of singularities of pairs preserving semi-simple normal crossings, 
Rev. R. Acad. Cienc. Exactas F\'is. Nat. Ser. A Math. RACSAM 
{\textbf{107}} (2013), no. 1, 159--188.

\bibitem[F1]{fujino-higher} 
O.~Fujino, 
Higher direct images of log canonical divisors, 
J. Differential Geom. {\textbf{66}} (2004), no. 3, 453--479.

\bibitem[F2]{fujino-unpublished}
O.~Fujino, Introduction to the minimal model program for 
log canonical pairs, preprint (2008). 

\bibitem[F3]{fujino-on-injectivity} 
O.~Fujino, 
On injectivity, vanishing and torsion-free theorems for algebraic varieties, 
Proc. Japan Acad. Ser. A Math. Sci. {\textbf{85}} (2009), no. 8, 95--100. 

\bibitem[F4]{fujino-introduction}
O.~Fujino, 
Introduction to the theory of quasi-log varieties, 
{\em{Classification of algebraic varieties}}, 289--303, 
EMS Ser. Congr. Rep., Eur. Math. Soc., Z\"urich, 2011. 

\bibitem[F5]{fujino-funda} 
O.~Fujino, Fundamental theorems for the log minimal model program, 
Publ. Res. Inst. Math. Sci. {\textbf{47}} (2011), no. 3, 727--789.

\bibitem[F6]{fujino-slc} 
O.~Fujino, 
Fundamental theorems for semi log canonical pairs, 
Algebr. Geom. {\textbf{1}} (2014), no. 2, 194--228. 

\bibitem[F7]{fujino-vanishing}
O.~Fujino, Vanishing theorems, 
to appear in Adv. Stud. Pure Math. 

\bibitem[F8]{fujino-injectivity} 
O.~Fujino, Injectivity theorems, to appear in Adv. Stud. Pure Math. 

\bibitem[F9]{fujino-pull} 
O.~Fujino, Pull-back of quasi-log structures, preprint (2013). 

\bibitem[F10]{fujino-reid-fukuda} 
O.~Fujino, 
Basepoint-free theorem of Reid--Fukuda type for quasi-log schemes, 
preprint (2014). 

\bibitem[F11]{fujino-foundation}
O.~Fujino, 
Foundation of the minimal model program, 
preprint (2014), 2014/4/16 version 0.01.

\bibitem[FF]{fujino-fujisawa} 
O.~Fujino, T.~Fujisawa, 
Variations of mixed Hodge structure and semipositivity theorems, 
Publ. Res. Inst. Math. Sci. {\textbf{50}} (2014), no. 4, 589--661.

\bibitem[Kl]{kleiman} 
S.~L.~Kleiman, 
Misconceptions about $K_X$,  
Enseign. Math. (2) {\textbf{25}} (1979), no. 3-4, 203--206 (1980). 

\bibitem[Ko]{kollar} 
J.~Koll\'ar, 
{\em{Singularities of the minimal model program}}, 
Cambridge Tracts in Mathematics, {\textbf{200}}. Cambridge 
University Press, Cambridge, 2013.

\bibitem[KM]{kollar-mori} 
J.~Koll\'ar, S.~Mori, 
{\em{Birational geometry of algebraic varieties}}, 
Cambridge Tracts in Mathematics, {\textbf{134}}. Cambridge 
University Press, Cambridge, 1998.

\bibitem[KSS]{kss} 
S.~J.~Kov\'acs, K.~Schwede, K.~E.~Smith, 
The canonical sheaf of Du Bois singularities, 
Adv. Math. {\textbf{224}} (2010), no. 4, 1618--1640.

\bibitem[L]{liu} 
Q.~Liu, {\em{Algebraic geometry and arithmetic curves}}, 
Oxford Graduate Texts in Mathematics, {\textbf{6}}. Oxford 
Science Publications. Oxford University Press, Oxford, 2002.
\end{thebibliography}
\end{document}